\documentclass[11pt,reqno,a4paper]{amsart}

\usepackage{mathtools}
\usepackage{etoolbox}
\usepackage{amsmath}
\usepackage{enumerate}
\usepackage{amssymb}
\usepackage{amscd} 
\usepackage{amsthm}
\usepackage{amsfonts}
\usepackage{cancel}
\usepackage[usenames,dvipsnames,svgnames]{xcolor}
\usepackage{tikz}
\usepackage{graphicx}
\usepackage[matrix, arrow, curve]{xy}
\usepackage{comment}
\usepackage[scaled]{helvet} % ss                                                     
\usepackage{courier} % tt        
\usepackage[T1]{fontenc}                                                    
\usepackage{hyperref}
\usepackage{imakeidx}
\usepackage{extarrows}
\usepackage[utf8]{inputenc}

\numberwithin{equation}{section}

\newcommand{\cU}{\mathcal{U}}

\newcommand{\N}{\mathbb{N}}

\newcommand{\R}{\mathbb{R}}
\newcommand{\Z}{\mathbb{Z}}

\newcommand{\qand}{\quad \textrm{and} \quad}

\newcommand\subsetsim{\mathrel{%
\ooalign{\raise0.2ex\hbox{$\subset$}\cr\hidewidth\raise-0.8ex\hbox{\scalebox{0.9}{$\sim$}}\hidewidth\cr}}}

\newcommand{\asdim}{\operatorname{asdim}}

\newcommand{\diam}{\operatorname{diam}}

\theoremstyle{thm}
\newtheorem{theorem}{Theorem}[section]

\newtheorem{proposition}[theorem]{Proposition}
\newtheorem{lemma}[theorem]{Lemma}

\theoremstyle{definition}
\newtheorem{definition}[theorem]{Definition}

\newtheorem{remark}[theorem]{Remark}
\newtheorem{example}[theorem]{Example}

\makeindex

\patchcmd{\section}{-.5em}{.5em}{}{}
\patchcmd{\subsubsection}{-.5em}{.5em}{}{}

\makeatletter
\@namedef{subjclassname@2020}{%
  \textup{2020} Mathematics Subject Classification}
\makeatother

\makeatletter
\let\@wraptoccontribs\wraptoccontribs
\makeatother

\begin{document}

\title[Hurewicz formula for asdim-lowering symm.\ quasimorphisms of approx.\ groups]{A Hurewicz-type formula for asymptotic-dimension-lowering symmetric quasimorphisms of countable approximate groups}

%    Author information
\author{Vera Toni\'c}
\address{Faculty of Mathematics, University of Rijeka, Croatia}
\curraddr{}
\email{vera.tonic@math.uniri.hr}
\thanks{This work was supported by the University of Rijeka project \emph{uniri-iskusni-prirod-23-66}.}



%\date{\today}

\begin{abstract} 
A well-known Hurewicz-type formula for asymptotic-dimension-lowering group homomorphisms, due to A.\ Dranishnikov and J.\ Smith, states that if $f:G\to H$ is a group homomorphism, then $\asdim G \leq \asdim H + \asdim\ \! (\ker f)$. In this paper we establish a similar formula for certain quasimorphisms of countable approximate groups: if $(\Xi, \Xi^\infty)$ and $(\Lambda, \Lambda^\infty)$ are countable approximate groups and if $f:(\Xi, \Xi^\infty)\to (\Lambda,\Lambda^\infty)$ is a symmetric unital quasimorphism, we show that $\asdim \Xi \leq \asdim \Lambda + \asdim\ \! (f^{-1}\!(D(f)))$, where $D(f)$ is the defect set of $f$.
\end{abstract}

\maketitle

\section{Introduction}

In classical dimension theory in topology, there is a well-known theorem stating that for a closed map $f:X\to Y$ between metric spaces it is true that $\dim X\leq \dim Y + \dim f$, where $\dim f$ is the supremum of dimensions of the fibers $f^{-1}(y)$, for all $y\in Y$. This theorem was proven independently by K.\ Morita in 1956 and by K.\ Nagami in 1957, but it is known as Hurewicz dimension-lowering mapping theorem, since the earliest version of it, for compact metric spaces, was proven by W.\ Hurewicz in 1927, and then extended to separable metric spaces by W.\ Hurewicz and H.\ Wallman in 1941 (according to \cite{Engelking}). 

This theorem has inspired analogous theorems developed in several directions, like the versions for asymptotic dimension of metric spaces and Lipschitz or coarsely Lipschitz maps, quoted here as Theorem \ref{AbstractHurewicz0}  due to G.\ Bell and A.\ Dra\-nishnikov (\cite{BellDran-Hurewicz}), and Theorem \ref{AbstractHurewicz}  due to N.\ Brodskiy, J.\ Dydak, M.\ Levin and A.\ Mitra (\cite{BDLM}). A version known as Hurewicz-type formula for asymptotic dimension and group homomorphisms, quoted here as Theorem \ref{Thm: Hurewicz DS}  due to A.\ Dra\-nishnikov and J.\ Smith (\cite{DranSmith}), stating that for a group homomorphism $f:G\to H$, it is true that $\asdim G \leq \asdim H +\asdim \ \!(\ker f)$.

In \cite{Hartnick-Tonic}, T.\ Hartnick and the author of this paper have proven a Hurewicz-type formula, listed here as Theorem \ref{Thm: H-T}, stating that for a global morphism of countable approximate groups $f:(\Xi, \Xi^\infty)\to (\Lambda,\Lambda^\infty)$, it is true that $\asdim \Xi\leq \asdim \Lambda + \asdim\ \! [\![ \ker f]\!]_c$. Our goal in this paper is to show that a similar formula works for a more general sort of function between countable approximate groups. Namely, we prove in Theorem \ref{Main-Hurewicz-quasimorphism} that if $f:(\Xi, \Xi^\infty)\to (\Lambda,\Lambda^\infty)$ is a symmetric unital \emph{quasimorphism} between countable approximate groups, then $\asdim \Xi\leq \asdim \Lambda + \asdim \ \!(f^{-1}(D(f)))$, where $D(f)$ is the \emph{defect set} of the quasimorphism $f$.

\medskip

We will define approximate groups, as well as global morphisms and quasimorphisms between them in Section 3 of this paper, where we will also define asymptotic dimension of countable approximate groups. Before going into Section 3, Section 2 will contain a short reminder about basic definitions of asymptotic dimension on metric spaces, coarse equivalences and coarsely Lipschitz maps.  In Section 4 we will list all versions of Hurewicz-type theorems that are relevant for us, finishing with the statement of our main theorem, Theorem \ref{Main-Hurewicz-quasimorphism}, which we will prove in the final section, Section 5.

%section 2
\section{Coarsely Lipschitz maps, coarse equivalence and asymptotic dimension}

In a metric space $(X,d)$ we will use notation $B(x, r)$ for an open ball, and $\overline{B}(x, r)$ for a closed ball with center at the point $x$ and with radius $r>0$. If $A\subseteq X$ and $R>0$,  then $N_R(A)$ will refer to an open $R$-neighborhood of a set $A$ in $(X,d)$. We will use the words \emph{function} and \emph{map} interchangeably.
Our main theorem refers to asymptotic dimension and in the proof of it we use Theorem \ref{AbstractHurewicz}, which mentions coarsely Lipschitz\footnote{The terminology we are using for maps is from \cite{CdlH}, also used in \cite{CHT}.} maps. We will also need the notion of coarse equivalence of metric spaces, so let us now review all of these definitions.

\begin{definition}\label{def: coarsely - maps} 
Let $(X, d_X)$, $(Y, d_Y)$ be metric spaces and let $f: X \to Y$ be a map. Let $\Phi_-, \ \Phi_+: [0, \infty) \to [0, \infty)$ be non-decreasing functions with $\lim_{t \to \infty} \Phi_i(t) = \infty$, for $i=-,+$.
\begin{enumerate}[(i)]
\item If $d_Y(f(x), f(x')) \leq \Phi_+(d_X(x,x'))$ for all $x,x'\in X$, then $f$ is called \emph{coarsely Lipschitz}, and $\Phi_+$ is referred to as an \emph{upper control function} for $f$.

\item If $\Phi_-(d_X(x,x')) \leq d_Y(f(x), f(x')) \leq \Phi_+(d_X(x,x'))$ for all $x,x'\in X$, then $f$ is called a \emph{coarse embedding}.

\item If $f$ is a coarse embedding which is also \emph{coarsely surjective}, i.e., if there is an $L>0$ such that $N_L(f(X))=Y$, then $f$ is called a \emph{coarse equivalence}.
\end{enumerate}
\end{definition}

We say that two metric spaces are \emph{coarsely equivalent}, and write $(X,d_X)\overset{CE}\approx (Y, d_Y)$, if there is a map between them which is a coarse equivalence. Coarsely Lipschitz maps can be characterized without mentioning control functions (see \cite[Prop.\ 3.A.5]{CdlH}):
\begin{lemma} \label{characterization-CL}
A map $f:X\to Y$ between metric spaces is coarsely Lipschitz if and only if for each $t> 0$, there exists $s> 0$ such that, if $x, x' \in X$ satisfy $d_X(x, x')\leq t$, then $d_Y (f(x),f(x'))\leq s$. 
\end{lemma}

Now let us give a definition of asymptotic dimension of a metric space, for which several other equivalent definitions can be found in \cite{BellDran1}.

\begin{definition}\label{def: asdim} Let $n \in \mathbb N_0$.
A metric space $(X,d)$ has \emph{asymptotic dimension} $\asdim X = n$ if this $n$ is the smallest number for which the following is true: for every $R >0$ there is a cover $\cU$ of $X$ such that:
\begin{enumerate}
\item  $\cU$ can be written as a union of $n+1$ collections $\mathcal U^{(0)}, \dots, \mathcal U^{(n)}$ of subsets of $X$, i.e., $\cU=\bigcup_{i=0}^n \mathcal U^{(i)}$, so that each of $\mathcal U^{(i)}$ is \emph{$R$-disjoint}, that is, whenever $U,V \in \mathcal U^{(i)}$ are such that $U\neq V$, then $d(U, V) \geq R$, and
\item $\cU$ is \emph{uniformly bounded}, i.e., there exists $D>0$ such that $\diam U \leq D$, for all $U \in \cU$.
\end{enumerate}
If there is no such $n\in \N_0$, we say that $\asdim X = \infty$.
\end{definition}

Among the properties of $\asdim$ (to be found in  \cite{BellDran1}), we particularly need the following one:
\begin{lemma}
If $(X,d_X)$ and $(Y, d_Y)$ are coarsely equivalent metric spaces, then $\asdim X =\asdim Y$, i.e., $\asdim$ is a coarse invariant.
\end{lemma}

We will also need the notion of $\asdim$ being \emph{uniformly bounded on a collection of spaces} (\cite[Section 2]{BellDranOnAsdimOfGroups}):

\begin{definition} Let $\mathbb Y \coloneqq  \{Y_\alpha\}_{\alpha \in A}$ be a collection of metric spaces and let $n\in \N_0$. We say that the asymptotic dimension of $\mathbb Y$ is \emph{uniformly bounded by $n$}, and write $\asdim\ \!(\mathbb Y) \overset{u}{\leq} n$, if for any $R > 0$ there exists $D>0$ such that for every $\alpha \in A$ there is a cover $\cU_\alpha$ of $Y_\alpha$ which satisfies the two properties from Definition \ref{def: asdim} with respect to $R$ and $D$.
\end{definition}

%section 3
\section{Approximate groups and quasimorphisms}
Before introducing the needed definitions, we need some technical facts.
Let us recall that a metric space is \emph{proper} if its closed balls are compact. 
In what follows, we will need to choose ``nice'' metrics on countable groups, so let us note that we always consider countable groups as discrete groups.

\begin{remark}\label{left-inv proper metric}
Recall that on a countable group (which need not be finitely generated), one can always choose a left-invariant proper metric so that this metric agrees with discrete topology on the group (see, for example, \cite{DranSmith}, Section 1).
Since this metric is proper, closed balls in it are compact and so they are finite sets, being in a discrete group. Therefore, in this setting, open balls with bounded radii are also finite sets.
\end{remark}

 Next we need to introduce some notation. 
 Let $A$ and $B$ be subsets of a group $(G,\cdot)$. Then $AB:=\{a b\ | \ a\in A, b\in B\}$,  so $A^2=AA=\{ab\ | \ a,b\in A\}$, and $A^k=A^{k-1}A$, for $k\in\N_{\geq 2}$. We will use $A^{-1}:=\{a^{-1}\ | \ a\in A\}$, and if $A=A^{-1}$, we say that $A$ is \emph{symmetric}.
We mark the identity element of the group $G$ by $e$ or $e_G$, and if $e\in A$, we say that $A$ is \emph{unital}. Also if $g\in G$, then 
$gA=g\cdot A:=\{g a\ | \ a\in A\}$, where we will make the operation sign visible if it helps better understand what is written.

%subsection 3.1
\subsection{Approximate subgroups and approximate groups}
Now let us introduce the notion of an approximate subgroup of a group. The following definition is due to T.~Tao (\cite{Tao}):

\begin{definition}[Approximate subgroup of a group]\label{DefTao} Let $G$ be a group and let $k \in \N$. A subset $\Lambda \subset G$ is called 
a \emph{$k$-approximate subgroup} of $G$ if
\begin{enumerate}[({AG}1)]
\item $\Lambda = \Lambda^{-1}$ and $e \in \Lambda$, i.e., $\Lambda$ is symmetric and unital, and
\item there exists a finite subset $F \subset  G$ such that $\Lambda^2\subset \Lambda  F$ and $|F| = k$.
\end{enumerate}
We say that $\Lambda$ is an \emph{approximate subgroup} of $G$ if it is a $k$-approximate subgroup for some $k \in \N$. 
\end{definition}

We will be interested in countably infinite approximate subgroups, so as long as the set $F$ is finite, the number of its elements is not going to be important to us.
The idea behind introducing approximate subgroups is allowing for the result of the group operation between two elements of a subset to be outside of this subset, but still ``a finite set away''.
Since for an approximate subgroup $\Lambda$ of a group $G$ there is the smallest subgroup $\Lambda^\infty := \bigcup_{k \in \N} \Lambda^k$ of $G$ which contains $\Lambda$, this leads to the following definition:

\begin{definition}[Approximate group]\label{DefApGr}
If $\Lambda$ is an approximate subgroup of a group $G$, then the group $\Lambda^\infty = \bigcup_{k \in \N} \Lambda^k$, which is the smallest subgroup of $G$ containing $\Lambda$, is called the \emph{enveloping group} of $\Lambda$. The pair $(\Lambda, \Lambda^\infty)$ is called an \emph{approximate group} and the associated filtered group $(\Lambda^\infty, (\Lambda^k)_{k \in \N})$ is called a \emph{filtered approximate group}.
\end{definition}
We say that an approximate group $(\Lambda, \Lambda^\infty)$ is \emph{finite} if $\Lambda$ is  finite (but clearly $\Lambda^\infty$ need not be finite).
We say $(\Lambda, \Lambda^\infty)$ is \emph{countable} if $\Lambda$ is countable, which also implies that $\Lambda^\infty$ is countable.

Although the notation $(\Lambda,\Lambda^\infty)$ may look a bit cumbersome,  when we need to introduce a ``nice'' metric on $\Lambda$  in Definition \ref{def: canonical coarse class}, we will first do so on $\Lambda^\infty$, so we might as well mention $\Lambda^\infty$ next to $\Lambda$. On the other hand, when we define the asymptotic dimension of an approximate group (see Definition \ref{def: asdim-L}), we will write (just) $\asdim \Lambda$.

For a few detailed examples and some non-examples of approximate (sub)groups, we refer the reader to
\cite{CHT}, while a short list  of examples can be found in \cite{Hartnick-Tonic}.

\medskip

We will be interested in countable approximate groups, which we will consider with a left-invariant proper metric on their enveloping groups. Note that \cite[Prop. 1.1]{DranSmith} for countable groups gives us:
\begin{proposition} If $d$ and $d'$ are two left-invariant proper metrics on the same countable group $G$, then the identity map between 
$(G, d)$ and $(G, d')$ is a coarse equivalence. 
\end{proposition}

Moreover, by the first part of Lemma 3.1 of \cite{CHT}, we get:
\begin{lemma}\label{ExternalQIType} Let $G$ be a countable group, $A \subseteq G$ be a subset and $d$ and $d'$ be left-invariant proper metrics on $G$. Then the identity map from $(A, d|_{A \times A})$ to $(A, d'|_{A \times A})$ is a coarse equivalence. 
\end{lemma}

This allows us to introduce:
\begin{definition}\label{def: canonical coarse class}
The \emph{canonical coarse class $[G]_c$} of a countable group $G$ is the coarse equivalence class
of the metric space $(G, d)$, where $d$ is some (hence any) left-invariant proper metric on $G$, that is, 
$[G]_c:=[(G, d)]_c=\{(X,d')\ |\ (X,d') \text{ is a metric space such that } (X,d')\overset{CE}{\approx} (G, d)\}.$
\medskip

Let $(\Lambda, \Lambda^\infty)$ be a countable approximate group. Then for any subset $A \subseteq \Lambda^\infty$, we define the \emph{cannonical coarse class of $A$} by
\[
[A]_c \coloneqq  [(A, d|_{A \times A})]_c = \{(X,d')\ |\ (X,d') \text{ is a metric space such that } (X,d')\overset{CE}{\approx} (A, d|_{A \times A})\},
\]
where $d$ is some (hence any) left-invariant proper metric on the countable group $\Lambda^\infty$.
In particular, this defines the canonical coarse class of $\Lambda$,
$[\Lambda]_c \coloneqq  [(\Lambda, d|_{\Lambda \times \Lambda})]_c$.
\end{definition}

Also note that the canonical coarse class of $\Lambda$ is independent of the ambient group used to define it, because
if $\Lambda$ is an approximate subgroup of a countable group $G$ and $d$ is a left-invariant proper metric on $G$, then $d|_{\Lambda^\infty \times \Lambda^\infty}$ is a left-invariant proper metric on $\Lambda^\infty$ (which is contained in $G$), and hence
$[\Lambda]_c =  [(\Lambda, d|_{\Lambda \times \Lambda})]_c.$
Moreover, since the restriction of a left-invariant proper metric on $\Lambda^\infty$ is still a proper metric on $\Lambda$, the class $[\Lambda]_c$ admits a representative which is a proper metric space. 

Let us note here that for any left-invariant proper metric $d$ on $\Lambda^\infty$, we will refer to the metric $d|_{\Lambda\times\Lambda}$ as a \emph{canonical metric} on $\Lambda$.

\medskip

Now we define the asymptotic dimension of a countable group, and of a countable approximate group, as follows:
\begin{definition}\label{def: asdim - group}\label{def: asdim-L}
For a countable group $G$, its asymptotic dimension is defined as  
\[\asdim G := \asdim \ \! (G, d),\] 
where $d$ is any left-invariant proper metric on $G$. We can also define $\asdim \ \! ([G]_c) := \asdim \ \! (G, d)$, so $\asdim G =\asdim \ \! (G, d)=\asdim \ \! ([G]_c)$.

For a countable approximate group $(\Lambda, \Lambda^\infty)$, its asymptotic dimension is defined as
\[\asdim \Lambda := \asdim \ \! (\Lambda, d|_{\Lambda \times \Lambda}),
\]
where $d$ is any left-invariant proper metric on $\Lambda^\infty$. We can also define $\asdim \ \! ([\Lambda]_c) := \asdim \ \! (\Lambda, d|_{\Lambda\times\Lambda})$.
More generally, if $A$ is any subset of $\Lambda^\infty$,  then the asymptotic dimension of $A$ is defined as $\asdim A := \asdim \ \! (A, d|_{A \times A})=\asdim\ \! [A]_c$.
\end{definition}

%subsection 3.2
\subsection{Global morphisms and quasimorphisms of approximate groups}
Before giving the announced definitions, here are some basic notions we need.
If $G$ and $H$ are groups, $A$ is a symmetric subset of $G$ and $f:A\to H$ is a (set-theoretic) function, we say that $f$ is \emph{symmetric} if $f(a^{-1})=f(a)^{-1}$. If $A$ is unital, we say that $f:A\to H$ is \emph{unital} if $f(e_G)=e_H$. If for all $a_1,a_2\in A$ such that $a_1a_2\in A$ we have $f(a_1a_2)=f(a_1)f(a_2)$, we call such $f$ a \emph{partial homomorphism}.

We would now like to introduce functions between groups which, as S.\ Ulam suggested in \cite{Ulam}, do not satisfy a strict rule like $f(xy)=f(x)f(y)$, but instead satisfy such a rule ``approximately''. So instead of demanding, for $f:G\to H$, that $f(xy)=f(x)f(y)$, or, equivalently, $f(y)^{-1} f(x)^{-1} f(xy)=e_H$ for all $x,y\in G$, we loosen this requirement as follows:

\begin{definition}\label{DefQM} Let $G$ and $H$ be groups. A (set-theoretic) function $f: G \to H$ is called a \emph{quasimorphism} if its \emph{defect set}
\begin{equation}\label{LDefectSet}
D(f) \coloneqq  \{f(y)^{-1}f(x)^{-1}f(xy) \mid x, y \in G\}
\end{equation}
 is finite.
\end{definition}
\begin{remark} What we call a quasimorphism in Definition \ref{DefQM} should be called a \emph{left-quasi\-mor\-phism}, while $D(f)$ should be called a \emph{left-defect set}, and we should define a \emph{right-quasimorphism} by demanding that the \emph{right-defect set}
\begin{equation}\label{RDefectSet}
D^*(f) \coloneqq  \{f(x)f(y)f(xy)^{-1} \mid x, y \in G\}
\end{equation}
be finite. However, it was proved by N.~Heuer in \cite[Prop.\ 2.3]{Heuer1} that the two notions coincide.
\end{remark}

\begin{remark}
We are using the name quasimorphism in Definition \ref{DefQM} following \cite{CHT}. In other sources, like \cite{FujiKap} or \cite{Heuer1}, the same kind of function is called a \emph{quasihomomorphism}, while the name quasimorphism is reserved for this kind of function which has $\R$ or $\Z$ as its codomain.
\end{remark}

\begin{remark}[Properties of defect set]\label{Drho} 
Since $f(xy) = f(x)f(y)f(y)^{-1}f(x)^{-1}f(xy)$, the set $D(f)$ has the following properties:
\begin{equation}\label{Drho1}
f(xy) \in f(x)f(y)D(f) \qand f(x)f(y) \in f(xy)D(f)^{-1} \quad \text{ for all }x,y \in G.
\end{equation}
\end{remark}

For the particular case when $H$ is a countable group with a left-invariant proper metric $d$, we have, according to Proposition 2.47 in \cite{CHT}:
\begin{proposition}
If $G$ is a group and $H$ is a countable group with a left-invariant proper metric $d$, then a function $f:G\to H$ is a quasimorphism if and only if there exists a constant $C\geq 0$ such that 
\[
d(f(xy),f(x)f(y))\leq C, \ \ \text{ for all } x,y \in G.
\]
\end{proposition}

\begin{example}[Examples of quasimorphisms]\label{ex: QM}
Basic examples of quasimorphisms between groups are homomorphisms, as well as all functions $f:G\to H$ between groups that have finite image. As noted in \cite{Heuer1}, different quasimorphisms can be constructed as follows: if $H$ is a group containing an infinite cyclic subgroup $C$ and $\tau:\Z \to H$ is a homomorphism with $\tau(\Z)=C$, then for any quasimorphism $\phi: G\to \Z$, the composition $\tau\circ\phi:G\to H$ is a quasimorphism.

The so-called \emph{counting quasimorphisms} introduced by Brooks (\cite{Brooks}), from non-abelian free groups $F_r$ of rank $r$ to $\Z$, are examples of symmetric quasimorphisms, which, roughly speaking, assign to a reduced word from $F_r\setminus \{e\}$ the difference between the number of appearances of some smaller word in it minus the number of appearances of the inverse of this smaller word (see, for example, \cite[Example B.43]{CHT} for more details). More facts on real-valued quasimorphisms can be found in \cite[Appendix B.5]{CHT}.

More examples and properties of quasimorphisms with discrete groups as codomains can be found in \cite{FujiKap}, while \cite{BrandVerbitsky} covers a more general theory of so-called Ulam quasimorphisms.
\end{example}

We see in Example \ref{ex: QM} that there are symmetric quasimorphisms, and more about the importance of these can be found in Appendix B.$5$ of \cite{CHT}. If the codomain of a quasimorphism $f$ is $\Z$ or $\R$, then $f$ being symmetric implies that $f$ is also unital, since from $f(e)=f(e^{-1})=c$ and $f(e^{-1})=f(e)^{-1}=-c$ one gets $2c=0$, so $c=0$. But even if the codomain is not $\Z$ nor $\R$, we can still make a quasimorphism unital using the following proposition (\cite[Proposition 2.7]{Heuer1}):
\begin{proposition}
Let $f:G\to H$ be a quasimorphism. Then the map $\overline f:G\to H$ defined by $\overline f|_{G\setminus \{e_G\}} = f$ and $\overline f (e_G)=e_H$ is also a quasimorphism.
\end{proposition}

Now let us move on to global morphisms and quasimorphisms between approximate groups.

\begin{definition}
A \emph{global morphism} $f: (\Xi, \Xi^\infty) \to (\Lambda, \Lambda^\infty)$ between approximate groups is a group homomorphism $f: \Xi^\infty \to \Lambda^\infty$ which restricts to partial homomorphisms $f_k:=f|_{\Xi^k} :\Xi^{k}\to \Lambda^k$ for each $k\in \N$, that is, for all $\xi_1, \xi_2 \in \Xi^k$ which satisfy $\xi_1\xi_2 \in \Xi^k$ we have $f_k(\xi_1\xi_2) = f_k(\xi_1)f_k(\xi_2)$.
\end{definition}

Note that for a global morphism $f(\Xi)\subset \Lambda$ implies that $f(\Xi^k)=f(\Xi)^k\subset \Lambda^k$, for all $k \in \N_{\geq 2}$.

\begin{definition} \label{def: QM}
Let $(\Xi, \Xi^\infty)$ and $(\Lambda, \Lambda^\infty)$ be approximate groups.
 A function of pairs $f: (\Xi, \Xi^\infty) \to (\Lambda, \Lambda^\infty)$ is called a \emph{global quasimorphism} (or simply a \emph{quasimorphism}) if $f: \Xi^\infty \to \Lambda^\infty$ is a quasimorphism in the sense of Definition \ref{DefQM}.
\end{definition}

Clearly a global morphism of approximate groups is an example of a symmetric and unital quasimorphism between them.

From Definition \ref{def: QM} we see that a
 quasimorphism $f$ of approximate groups satisfies $f(\Xi)\subset\Lambda$, and we use notation $f_1:=f|_\Xi:\Xi \to \Lambda$, as we do for global morphisms.
But $f(\Xi^2)$ need not be contained in $\Lambda^2$, since $f(\xi_1\xi_2)$ is not equal to $f(\xi_1)f(\xi_2)$ in general. However, since $D(f)$ is finite, from \eqref{Drho1} we get that $f(\Xi^2)\subset \Lambda^M$, for some $M \in \N_{\geq 2}$.
 Let us also note that, if we wish for $f(\Xi)$ to be equal to $\Lambda$, we will need this quasimorphism $f$ to be symmetric and unital:

\begin{proposition}
\label{ImagesExist} Let $f: G \to H$ be a quasimorphism between groups. If $\Xi$ is an approximate subgroup of $G$ and $f$ is symmetric and unital, then $f(\Xi) $ is an approximate subgroup of $H$.
\end{proposition}
\begin{proof} First note that the image of a symmetric unital set under a symmetric unital map is symmetric and unital, so $f(\Xi)$ is symmetric and unital.
Secondly, since there exists a finite set $F\subset G$ such that $\Xi^2\subset \Xi F$, 
by \eqref{Drho1} we have 
\[
f(\Xi^2)\subset f(\Xi F) \subset f(\Xi) f(F)D(f).
\]
This, together with another application of \eqref{Drho1}, yields
\[
f(\Xi)^2 \subset f(\Xi^2)D(f)^{-1} \subset f(\Xi F)D(f)^{-1} \subset f(\Xi) f(F)D(f)D(f)^{-1},
\]
and the set $f(F)D(f)D(f)^{-1}$ is finite.
\end{proof}

%section 4
\section{The Hurewicz dimension-lowering mapping theorem and some existing generalizations}

The classical Hurewicz theorem for dimension-lowering maps (e.g.~\cite[Theorem 4.3.4]{Engelking}, due to Morita and Nagami), states:
\begin{theorem}[Hurewicz dimension-lowering theorem] Let $X$ and $Y$ be me\-tri\-zable spaces and let $f: X \to Y$ be a closed map. Then
\[
\dim X \leq \dim Y + \dim f, \quad \text{where} \quad \dim f \coloneqq  \sup\ \!\{ \dim(f^{-1}(y))\ | \  y\in Y\}.
\]
\end{theorem}

A version of this theorem for asymptotic dimension, due to Bell and Dranishnikov 
(see \cite[Theorem 1]{BellDran-Hurewicz} or \cite[Theorem 29]{BellDran1}), states:
\begin{theorem}[Asymptotic Hurewicz mapping theorem, first version]\label{AbstractHurewicz0} Let $f: X \to Y$ be a  Lipschitz map from a geodesic metric space $X$ to a metric space $Y$.  If for every $r>0$ the collection
$\mathbb Y_r \coloneqq  \{f^{-1}(B(y, r))\}_{y \in Y}$ satisfies 
$\asdim\ \!(\mathbb Y_r)  \overset{u}{\leq} n$,
then $\asdim X \leq \asdim Y + n.$
\end{theorem}

A generalization of this result, due to  Brodskiy, Dydak, Levin and Mitra (\cite[Theorem 1.2]{BDLM}), states:
\begin{theorem}[Asymptotic Hurewicz mapping theorem, second version]\label{AbstractHurewicz} Let $h: X \to Y$ be a coarsely Lipschitz map between metric spaces.  If for every $r>0$ the collection
$\mathbb Y_r \coloneqq  \{h^{-1}(B(y, r))\}_{y \in Y}$ satisfies
$\asdim\ \!(\mathbb Y_r)  \overset{u}{\leq} n$,
then $\asdim X \leq \asdim Y + n.$
\end{theorem}
The original statement of this theorem talks about a ``large scale uniform'' map, but this is precisely what is called a coarsely Lipschitz map in this paper in Definition \ref{def: coarsely - maps}. Also, instead of $\asdim X \leq \asdim Y + n$, in \cite[Theorem 1.2]{BDLM} it says  $\asdim X \leq \asdim Y +\asdim h$, where $\asdim h$ is defined  as $\sup \ \!\{\asdim A \mid A\subseteq X \text{ and } \asdim (h(A))=0\}$. However, the property $ \asdim(\mathbb Y_r)  \overset{u}{\leq} n$ can be shown to be equivalent to the map $h$ having an $n$-dimensional control function (terminology of \cite{BDLM}), and by Corollary 4.10 of \cite{BDLM}, this is equivalent to $\asdim h \leq n$.

\medskip

We follow with the Hurewicz-type formulas for homomorphisms of groups, and global morphisms of countable approximate groups. First of all, in (\cite[Theorem 2.3]{DranSmith}) Dranishnikov and Smith prove:
\begin{theorem}[Hurewicz-type formula for homomorphism of groups]\label{Thm: Hurewicz DS}
Let $f:G\to H$ be a homomorphism of groups. Then $\asdim G \leq \asdim H + \asdim\ \! (\ker f)$.
\end{theorem}
Then in \cite[Theorem 1.4]{Hartnick-Tonic}, relying on Theorem \ref{AbstractHurewicz}, it is proven:
\begin{theorem}[Hurewicz type formula for morphism of countable approximate groups]\label{Thm: H-T} Let $(\Xi, \Xi^\infty)$, $(\Lambda, \Lambda^\infty)$ be countable approximate groups and let $f: (\Xi, \Xi^\infty) \to (\Lambda, \Lambda^\infty)$ be a global morphism. Then
 \[
 \asdim\Xi \leq \asdim\Lambda + \asdim\ \!([\![\ker(f)]\!]_c),
 \]
where $[\![\ker(f)]\!]_c$ is the \emph{coarse kernel} of $f$. 
\end{theorem}
The coarse kernel $[\![\ker(f)]\!]_c$ of a global morphism $f$ is defined in \cite[Remark 4.12]{Hartnick-Tonic} as 
\[
[\![\ker (f)]\!]_c:=
[\Xi^2\cap\ker f]_c = [\Xi^3\cap \ker f]_c = \dots = [\Xi^k\cap \ker f]_c = [\Xi^{k+1}\cap \ker f]_c=\ldots,
\]
where these equalities are shown to be true by \cite[Corollary 4.11]{Hartnick-Tonic}.

\medskip

As announced in the Introduction, in this paper in Section 5 we are going to prove a Hurewicz-type formula for certain kind of quasimorphisms of countable approximate groups, namely:

\begin{theorem}\label{Main-Hurewicz-quasimorphism}
Let $(\Xi, \Xi^\infty)$ and $(\Lambda,\Lambda^\infty)$ be countable approximate groups, and let $f:(\Xi, \Xi^\infty) \to (\Lambda,\Lambda^\infty)$ be a symmetric unital quasimorphism. Then 
\[ \asdim \Xi \leq \asdim \Lambda + \asdim\ \! (f^{-1}(D(f))),\]
where $D(f)$ is the defect set of $f$.
\end{theorem}

%section 5
\section{Proof of the main theorem}
In the proof of Theorem \ref{Main-Hurewicz-quasimorphism} we intend to invoke Theorem \ref{AbstractHurewicz}, for which we need to make sure that $f_1=f|_\Xi:\Xi \to \Lambda$ is a coarsely Lipschitz map. 
The next lemma will show this. This statement is, in fact, true for all quasimorphisms between countable approximate groups (see \cite{CHT}, Lemma 3.8), but we will prove it only for symmetric quasimorphisms.
\begin{lemma}\label{coarsely-Lipschitz}
Let $(\Xi, \Xi^\infty)$ and $(\Lambda,\Lambda^\infty)$ be countable approximate groups, and let $f:(\Xi, \Xi^\infty) \to (\Lambda,\Lambda^\infty)$ be a symmetric quasimorphism. Then $f_1=f|_\Xi: \Xi \to \Lambda$ is coarsely Lipschitz, with respect to the canonical metrics on $\Xi$ and $\Lambda$.
\end{lemma}
\begin{proof}
First we fix a left-invariant proper metric $d$ on $\Xi^\infty$ and $d'$ on $\Lambda^\infty$. According to Lemma \ref{characterization-CL}, it suffices to show that, for any $t> 0$ there exists an $s> 0$ such that whenever $\xi, \eta \in \Xi$ satisfy $d(\xi, \eta)\leq t$, then $d'(f_1(\xi),f_1(\eta))\leq s$.

Let us take a random $t> 0$ and fix it. Since the metric $d$ is proper, the closed ball $\overline B(e_\Xi, t)=\{\xi \in \Xi^\infty \ | \ d(e_\Xi, \xi)\leq t\}$ is compact so it is finite, by Remark \ref{left-inv proper metric}. Thus $f(\overline B(e_\Xi, t))$ is a finite subset of $\Lambda^\infty$, so there exists an $S>0$ (depending on $t$) such that 
\begin{eqnarray}\label{finite balls}
f(\overline B(e_\Xi, t))\subset \overline{B'}(e_\Lambda, S):=\{\lambda \in \Lambda^\infty \ | \ d'(e_\Lambda, \lambda)\leq S\}.
\end{eqnarray}
Now let $\xi,\eta\in \Xi$ be any two elements with $d(\xi, \eta)\leq t$. Since $d'$ is left-invariant and $f$ is symmetric, it follows that
\begin{eqnarray}\label{triangle ineq}
d'(f_1(\xi), f_1(\eta)) &=& d'(f_1(\eta)^{-1}f_1(\xi), e_\Lambda)=d'(f_1(\eta^{-1})f_1(\xi), e_\Lambda)\nonumber \\
&\leq & d'(f_1(\eta^{-1})f_1(\xi), f(\eta^{-1}\xi) )+ d'(f(\eta^{-1}\xi), e_\Lambda).
\end{eqnarray}

Note that from $d(\xi,\eta)=d(\eta^{-1}\xi, e_\Xi)\leq t$ we have that $\eta^{-1}\xi$ is in $\overline B(e_\Xi, t)$, so by \eqref{finite balls} 
\begin{eqnarray}\label{basic}
d'(f(\eta^{-1}\xi), e_\Lambda)\leq S.
\end{eqnarray}

On the other hand, $d'(f_1(\eta^{-1})f_1(\xi), f(\eta^{-1}\xi) )=d'(f(\eta^{-1}\xi)^{-1}f_1(\eta^{-1})f_1(\xi), e_\Lambda)$. But by \eqref{Drho1} we have 
$f(\eta^{-1}\xi)^{-1}f_1(\eta^{-1})f_1(\xi)\in D(f)^{-1}$, and the defect set $D(f)$ is finite, so if we define
\[
C:=\max_{\lambda \in D(f)^{-1}} d'(\lambda, e_\Lambda),
\]
then 
\begin{eqnarray}\label{finish it off}
d'(f_1(\eta^{-1})f_1(\xi), f(\eta^{-1}\xi) )=d'(f(\eta^{-1}\xi)^{-1}f_1(\eta^{-1})f_1(\xi), e_\Lambda)\leq C.
\end{eqnarray}
Finally \eqref{triangle ineq}, \eqref{basic} and \eqref{finish it off} yield $d'(f_1(\xi),f_1(\eta))\leq C+S$, which finishes the proof.
\end{proof}

\medskip

Now we can prove the main theorem:
\begin{proof}[Proof of Theorem \ref{Main-Hurewicz-quasimorphism}]
By Proposition \ref{ImagesExist}, we may assume that $f|_\Xi=f_1:\Xi \to \Lambda$ is surjective, since we can replace $\Lambda$ with the approximate subgroup $f(\Xi)$, if needed. Also note that $e_\Lambda \in D(f)$, because $f$ is unital.

Let us fix left-invariant proper metrics $d$ on $\Xi^\infty$ and $d'$ on $\Lambda^\infty$. For an $r\in \R_{>0}$ we denote by $B(\xi, r)$ and $B'(\lambda, r)$ open balls with radius $r$ in $(\Xi^\infty,d)$ and $(\Lambda^\infty, d')$, respectively, centered at $\xi \in\Xi^\infty$, $\lambda\in \Lambda^\infty$. Note that by left-invariance of $d'$ we have $B'(\lambda,r)=\lambda B'(e_\Lambda,r)$, for all $\lambda \in \Lambda^\infty$.

Now by Lemma \ref{coarsely-Lipschitz}, the restriction $f_1: \Xi \to \Lambda$ is coarsely Lipschitz.
Because of Theorem \ref{AbstractHurewicz}, it suffices to show that for every $r>0$, the collection
\[
\mathbb Y_r \coloneqq  \{f_1^{-1}(B'(\lambda, r) \cap \Lambda)\}_{\lambda \in \Lambda}
\]
satisfies the inequality
\begin{equation}\label{HurewiczToShow1}
\asdim \ \! (\mathbb{Y}_r)  \overset{u}{\leq} {\asdim\ \! f^{-1}(D(f)) }.
\end{equation}

First, let us fix a random $r > 0$. Then, for every $\lambda \in \Lambda$ pick a $\xi_\lambda \in f_1^{-1}(\lambda)\subset\Xi$, so $f(\xi_\lambda)=f_1(\xi_\lambda)=\lambda$.
Furthermore, 
note that \[(\lambda B'(e_\Lambda,r))\cap\Lambda \subset \lambda \cdot\left(B'(e_\Lambda,r)\cap\Lambda^2\right),\] since if $z=\lambda b=\widetilde\lambda$, for some $b\in B'(e_\Lambda,r)$, $\widetilde\lambda \in\Lambda$, then $b=\lambda^{-1}\widetilde\lambda\in \Lambda^2$. Now, for any $\lambda\in\Lambda$,
\begin{eqnarray} \label{first long}
f_1^{-1}(B'(\lambda, r) \cap \Lambda) &=& f_1^{-1}((\lambda B'(e_\Lambda, r)) \cap \Lambda) \nonumber \\
 &\subset& f_1^{-1}(\lambda\cdot (B'(e_\Lambda, r)\cap \Lambda^2)) \nonumber \\
&=& f_1^{-1}\left( f_1\left( \xi_\lambda\right)\cdot (B'(e_\Lambda, r)\cap \Lambda^2)\right) \nonumber \\
&\subset&  f^{-1}\left(f(\xi_\lambda)\cdot\left(B'(e_\Lambda, r) \cap \Lambda^2\right)\right) \nonumber
\\
&=& \{z \in \Xi^\infty \mid f(z) \in  f(\xi_\lambda)\cdot(B'(e_\Lambda, r) \cap \Lambda^2)\} \nonumber \\
&=& \{z \in \Xi^\infty \mid f(\xi_\lambda)^{-1}f(z) \in B'(e_\Lambda, r) \cap \Lambda^2\} \nonumber\\
&=& \{z \in \Xi^\infty \mid f(\xi_\lambda^{-1})f(z) \in B'(e_\Lambda, r) \cap \Lambda^2\}.
\end{eqnarray}

By \eqref{Drho1} we know that $f(\xi_\lambda^{-1})f(z) \in f(\xi_\lambda^{-1} z) D(f)^{-1}$, so $f(\xi_\lambda^{-1})f(z) = f(\xi_\lambda^{-1} z)\cdot d^{-1}$ for some $d\in D(f)$. Thus 
\[
f(\xi_\lambda^{-1})f(z) \in B'(e_\Lambda, r) \cap \Lambda^2 \ \ \Rightarrow\ \  f(\xi_\lambda^{-1} z)\cdot d^{-1} \in B'(e_\Lambda, r) \cap \Lambda^2
\]
so 
\begin{eqnarray}\label{second long}
f(\xi_\lambda^{-1}z) \in \left(B'(e_\Lambda, r) \cap \Lambda^2\right) D(f) & \Rightarrow& \xi_\lambda^{-1}z \in f^{-1}\!\left(\left(B'(e_\Lambda, r) \cap \Lambda^2\right) D(f) \right) \nonumber \\
&\Rightarrow& z\in \xi_\lambda\cdot f^{-1}\!\left(\left(B'(e_\Lambda, r) \cap \Lambda^2\right) D(f) \right).
\end{eqnarray}

Therefore, from \eqref{first long} and \eqref{second long} we get
\begin{eqnarray}\label{up to iso}
f_1^{-1}(B'(\lambda, r) \cap \Lambda) &\subset&
\{z \in \Xi^\infty \mid f(\xi_\lambda^{-1})f(z) \in B'(e_\Lambda, r) \cap \Lambda^2\} \nonumber
\\
&\subset&  \xi_\lambda\cdot f^{-1}\!\left(\left(B'(e_\Lambda, r) \cap \Lambda^2\right) D(f) \right).
\end{eqnarray}
Since left-multiplication by $\xi_\lambda$ yields an isometry of $\Xi^\infty$, we have reduced proving \eqref{HurewiczToShow1} to proving that
\begin{equation}\label{HurewiczToShow2}
\asdim \ \!  f^{-1}\!\left(\left(B'(e_\Lambda, r) \cap \Lambda^2\right) D(f) \right) \leq \asdim \ \! f^{-1}(D(f)),
\end{equation}
for every $r > 0$. 

Now by properness of $d'$ the ball $B'(e_\Lambda, r)$ is finite, so both $B'(e_\Lambda, r)\cap\Lambda$ and $B'(e_\Lambda, r)\cap\Lambda^2$ are finite and we may write, using $e_\Lambda\in \Lambda\subset \Lambda^2$,
\[
B'(e_\Lambda, r)\cap\Lambda=\{\lambda_1=e_\Lambda, \lambda_2, \ldots , \lambda_N\}\subset\{\lambda_1, \ldots, \lambda_N, \lambda_{N+1}, \lambda_{N+2}, \ldots , \lambda_{N+k}\}= B'(e_\Lambda, r)\cap\Lambda^2.
\]
We also know that $D(f)$ is finite, so using $e_\Lambda\in D(f)$ we may write $D(f)=\{d_1=e_\Lambda, d_2, \ldots , d_m\}$.
It follows that
\[
\left(B'(e_\Lambda, r) \cap \Lambda^2\right) D(f) = \{\lambda_i \mid i=1, \ldots , N+k\} \cup \{ \lambda_id_j \mid i=1, \ldots , N+k, \ j=2, \ldots , m\}.
\]
Therefore 
\[
f^{-1}\!\left( \left(B'(e_\Lambda, r) \cap \Lambda^2\right) D(f) \right) = \left(\bigsqcup_{i=1}^{N}f^{-1}(\lambda_i)\right) \sqcup 
\left(\bigsqcup_{i=1}^{k}f^{-1}(\lambda_{N+i})\right) \sqcup \left( \bigsqcup_{\substack{i=1, \ldots , N+k\\ j=2, \ldots , m}} f^{-1}(\lambda_id_j) \right).
\]

Since $f|_\Xi=f_1:\Xi\to \Lambda$ is surjective, we know that $f^{-1}(\lambda_i)$ are non-empty for $i=1, \ldots , N$, though some of the other preimages in this union may be empty. For the sake of the argument, let us assume that all of the preimages mentioned in this union are non-empty (if some were empty, we could just remove them from further calculations). Let us take 
\[\xi_i\in \Xi^\infty \text{ so that } f(\xi_i)=\lambda_i, \text{ for } i=1, \ldots , N+k,   \ \text{ and}\]
\[\xi_{ij}\in \Xi^\infty \text{ so that } f(\xi_{ij})=\lambda_id_j, \text{ for  } i=1, \ldots , N+k \text{ and }j=2, \ldots , m.
\]
Define
\begin{eqnarray}\label{replace by inverses}
R \coloneqq  \max \ \! \left(\left\{d(e_\Xi, \xi_i) \mid\ \text{\footnotesize{\(i=1, \ldots , N+k\)}} \right\}\cup \left\{d(e_\Xi, \xi_{ij}) \mid \ \text{\footnotesize{\( i=1, \ldots , N+k , \ j=2, \ldots , m\)}}\right\}\right)
\end{eqnarray}
and note that, by left-invariance of the metric $d$, we get the same $R$ in \eqref{replace by inverses} if we replace each $\xi_i$ and $\xi_{ij}$ by its inverse. 

Now we claim that all $f^{-1}(\lambda_i)$ and all $f^{-1}(\lambda_id_j)$, for $i=1, \ldots , N+k$ and $j=2, \ldots , m$, are contained in $N_R(f^{-1}(D(f)))$. To see this, take any $\eta_i \in f^{-1}(\lambda_i)$, for $ i=1, \ldots , N+k$, and note that using \eqref{Drho1} and the fact that $f$ is symmetric we have
\[
f(\eta_i\xi_i^{-1})\in f(\eta_i)f(\xi_i^{-1})D(f)=f(\eta_i)f(\xi_i)^{-1}D(f)=\lambda_i\lambda_i^{-1}D(f)=D(f),
\]
so $\eta_i\xi_i^{-1}$ is contained in $f^{-1}(D(f))$.
Analogously, for any $\eta_{ij} \in f^{-1}(\lambda_id_j)$, for $ i=1, \ldots , N+k$, $j=2,\ldots , m$,  we have
\[
f(\eta_{ij}\xi_{ij}^{-1})\in f(\eta_{ij})f(\xi_{ij}^{-1})D(f)=f(\eta_{ij})f(\xi_{ij})^{-1}D(f)=\lambda_id_j(\lambda_id_j)^{-1}D(f)=D(f), 
\]
so $\eta_{ij}\xi_{ij}^{-1}$ is contained in $f^{-1}(D(f))$.
Therefore by \eqref{replace by inverses}
\[
d(\eta_i, f^{-1}(D(f)))\leq d(\eta_i, \eta_i\xi_i^{-1})\leq R, \text{ for } i=1, \ldots, N+k, \text{ and}
\]
\[
d(\eta_{ij}, f^{-1}(D(f)))\leq d(\eta_{ij}, \eta_{ij}\xi_{ij}^{-1})\leq R, \text{ for } i=1, \ldots, N+k, \ j=2, \ldots, m,
\]
which finishes the proof that all $f^{-1}(\lambda_i)$ and all $f^{-1}(\lambda_id_j)$ are contained in $N_R(f^{-1}(D(f)))$.

Thus the entire $f^{-1}\!\left( \left(B'(e_\Lambda, r) \cap \Lambda^2\right) D(f) \right)$ is contained in $N_R(f^{-1}(D(f)))$, so
\[\asdim \ \!  f^{-1}\!\left(\left(B'(e_\Lambda, r) \cap \Lambda^2\right) D(f) \right) \leq \asdim \ \!N_R( f^{-1}(D(f)))\]
and $N_R(f^{-1}(D(f)))$ has the same asymptotic dimension as $f^{-1}(D(f))$,
which establishes \eqref{HurewiczToShow2} and finishes the proof.
\end{proof}

\begin{remark}
Since asymptotic dimension is a coarse invariant, we could replace $f^{-1}(D(f))$ in the statement of Theorem \ref{Main-Hurewicz-quasimorphism} by its coarse class $[f^{-1}(D(f)) ]_c$.
\end{remark}

\end{document}